\documentclass{amsart}
\usepackage{amssymb,color,latexsym,amsmath,amsthm}
\usepackage{fullpage}
\usepackage[numbers,sort&compress]{natbib}
\numberwithin{equation}{section}

\begin{document}

\newtheorem*{thm*}{Theorem A}
\newtheorem{mthm}{Theorem}
\newtheorem{mcor}{Corollary}
\newtheorem{mpro}{Proposition}
\newtheorem{mfig}{figure}
\newtheorem{mlem}{Lemma}
\newtheorem{mdef}{Definition}
\newtheorem{mrem}{Remark}
\newtheorem{mpic}{Picture}
\newtheorem{rem}{Remark}[section]
\newcommand{\ra}{{\mbox{$\rightarrow$}}}
\newtheorem{thm}{Theorem}[section]
\newtheorem{pro}{Proposition}[section]
\newtheorem{lem}{Lemma}[section]
\newtheorem{defi}{Definition}[section]
\newtheorem{cor}{Corollary}[section]
\newcommand{\myref}[1]{A}

\numberwithin{equation}{section}
\newtheorem*{acknow}{Acknowlegement}

\title[]{Partial regularity and Liouville theorems for stable solutions of anisotropic elliptic equations}

\author{Mostafa Fazly }

\address{Department of Mathematics, The University of Texas at San Antonio, San Antonio, TX 78249, USA}
\email{mostafa.fazly@utsa.edu}

\author{Yuan Li }
\address{College of Mathematics and
Econometrics, Hunan University, Changsha 410082,
                 PRC and Department of Mathematics, The University of Texas at San Antonio, San Antonio, TX 78249, USA.  }
\email{liy93@hnu.edu.cn and yuan.li3@utsa.edu }

\thanks{This work is part of the second author's Ph.D. dissertation and supported by a China Scholarship Council funding.}

\subjclass[2010]{Primary 35B08; Secondary  35A01}

\keywords{Quasilinear equation, partial regularity, Liouville  theorem, stable solutions, Hausdorff dimension.}

\date{}

\begin{abstract}
We study the quasilinear elliptic equation
\begin{equation}
-Qu=e^u \indent\mbox{in}\indent \Omega\subset \mathbb{R}^{N},\nonumber
\end{equation}
where the operator $Q$, known as Finsler-Laplacian (or anisotropic Laplacian), is defined by
$$Qu:=\sum_{i=1}^{N}\frac{\partial}{\partial x_{i}}(F(\nabla u)F_{\xi_{i}}(\nabla u)),$$
where $F_{\xi_{i}}=\frac{\partial F}{\partial\xi_{i}}$ and $F: \mathbb{R}^{N}\rightarrow[0,+\infty)$ is a convex function of  $ C^{2}(\mathbb{R}^{N}\setminus\{0\})$, that satisfies certain assumptions. For bounded domain $\Omega$ and for a stable weak solution of the above equation, we prove that the Hausdorff dimension of singular set does not exceed $N-10$. 
 For the entire space, we apply Moser iteration arguments, established by Dancer-Farina and   Crandall-Rabinowitz in the context, to prove Liouville theorems for stable solutions  and for finite Morse index solutions in dimensions $N<10$ and $2<N<10$, respectively. We also provide an explicit solution that is stable outside a compact set in $N=2$. In addition, we provide similar Liouville theorems for the power-type nonlinearities. 
\end{abstract}

\maketitle

\section{Introduction and main results}
We study  stable weak solutions of the quasilinear Finsler-Liouville equation
\begin{equation}\label{1.1e}
-Qu=e^u \indent\mbox{in}\indent\Omega,
\end{equation}
where $\Omega$ is a subset of $\mathbb{R}^{N}$  and the operator $Q$ is defined by
$$Qu:=\sum_{i=1}^{N}\frac{\partial}{\partial x_{i}}(F(\nabla u)F_{\xi_{i}}(\nabla u)),$$
where $F_{\xi_{i}}:=\frac{\partial F}{\partial\xi_{i}}$ and $F: \mathbb{R}^{N}\rightarrow[0,+\infty)$ is a convex function of  $ C^{2}(\mathbb{R}^{N}\setminus\{0\})$ such that
 $F(t\xi)=|t|F(\xi)$ for any $t\in\mathbb{R}$ and $\xi\in\mathbb{R}^{N}$. The above equation is a particular case of  the quasilinear equation with nonlinearity $f\in C^1(\mathbb R)$,
\begin{equation}\label{1.1}
-Qu=f(u) \indent\mbox{in}\indent\Omega.
\end{equation}
 We assume that $F(\xi)>0$ for any $\xi\neq0$ and for such a function $F$, there exist constant $0<a\leq b<\infty$, $0<\lambda\leq\Lambda<\infty$ such that
\begin{equation}\label{1.2}
a|\xi|\leq F(\xi)\leq b|\xi|  \indent\mbox{for any}\indent \xi\in\mathbb{R}^{N},
\end{equation}
and
\begin{equation}\label{1.3}
\lambda^{2}|V|^{2}\leq F_{\xi_{i}\xi_{j}}(\xi)V_{i}V_{j}\leq \Lambda|V|^{2},
\end{equation}
for any $\xi\in\mathbb{R}^{N}$ and $V\in\xi^{\bot}$ where $\xi^{\bot}:=\{V\in\mathbb{R}^{N}: \langle V, \xi\rangle=0\}$. The operator $Q$ is known as anisotropic Laplacian or Finsler Laplacian operator in the literature.  When $F(\xi)=|\xi|$, that is the isotropic case, the operator $Q$ becomes the classical Laplacian operator. As an anisotropic Laplacian, such operators have been  studied vastly in the literature. In early twenty century,  Wulff \cite{W} used such operators to study crystal shapes and minimization of anisotropic surface tensions. The operator $Q$  is closely connected with a smooth,
convex hypersurface in $\mathbb R^N$, called the Wulff shape (or equilibrium crystal shape) of $F$.
The Wulff shape was introduced and studied by Wulff in \cite{W}.   In order to provide a few references in this context,
Wang and Xia in \cite{WX12} extended the classical result of Brezis and Merle \cite{BM91} to equation (\ref{1.1e}) in two dimensions. See also \cite{12} where the authors study an overdetermined problem for anisotropic equations.  Caffarelli et al. in  \cite{cgs} established gradient estimates and monotonicity formulae for quasilinear equations in order to study entire solutions, see also \cite{faz}. Cozzi et al. in \cite{2,cfv} proved such estimates and  formulae for singular, degenerate, anisotropic equations, see also \cite{FV14,LM}. We also refer interested readers to \cite{CS09, 6, 11} and references therein in the context of quasilinear equations.

What follows is the definition of weak and stable solutions of (\ref{1.1}).
\begin{defi}
We say that $u$ is a weak solution of (\ref{1.1}), if $u\in H^{1}_{loc}(\Omega)$ and $f(u)\in L^{1}_{loc}(\Omega)$, the following hold
$$\int_{\Omega}F(\nabla u)F_{\xi}(\nabla u)\cdot\nabla\phi dx=\int_{\Omega}f(u)\phi dx,$$
 for all $\phi\in C_{c}^{\infty}(\Omega)$.
\end{defi}

\begin{defi}
We say that the weak solution of equation (\ref{1.1}) is stable, if $f'(u)\in L^{1}_{loc}(\Omega)$, holds
\begin{equation}\label{1.5}
\int_{\Omega}F_{\xi_{i}}(\nabla u)F_{\xi_{j}}(\nabla u)\phi_{x_{i}}\phi_{x_{j}}+F(\nabla u)F_{\xi_{i}\xi_{j}}(\nabla u)\phi_{x_{i}}\phi_{x_{j}}-f'(u)\phi^{2}dx\geq0,
\end{equation}
for all $\phi\in C_{c}^{1}(\Omega)$.
\end{defi}

%%%%%%%%

Here we provide some properties and definitions related to the operator $Q$. Let $F^{0}$ be the support function of $K:=\{x\in\mathbb{R}^{N}: F(x)<1\}$ which is defined by
$$F^{0}(x):=\sup_{\xi\in K}\langle x, \xi\rangle.$$
We denote $B_{r}(x_{0}):=\{x\in\mathbb{R}^{N}: F^{0}(x-x_{0})<r\}$ a Wulff ball of radius $r$ with center at $x_{0}$, for convenience, we will use this notation $B_{r}(x_{0})$ throughout the paper. Denote $\kappa_{0}=|B_{1}(x_{0})|$, where $|B_{1}(x_{0})|$ is the Lebesgue measure of $B_{1}(x_{0})$.
 By the assumptions on $F$, one can see that the following properties hold.  Some of these are discussed in detail in \cite{6,12}.
\begin{thm*}
We have the following properties:

(1) $|F(x)-F(y)|\leq F(x+y)\leq F(x)+F(y)$;

(2) $|\nabla F(x)|\leq C$ for any $x\neq0$;

(3) $\langle \xi, \nabla F(\xi)\rangle=F(\xi)$, $\langle x, \nabla F^{0}(x)\rangle=F^{0}(x)$ for any $x\neq0$, $\xi\neq0$;

(4) $\sum_{j=1}^{N}F_{\xi_{i}\xi_{j}}(\xi)\xi_{j}=0$, for any $i=1,2,\cdots,N$;

(5) $F(\nabla F^{0}(x))=1$, $F^{0}(\nabla F(x))=1$;

(6) $F_{\xi_{i}}(t\xi)=sgn(t)F_{\xi_{i}}(\xi)$;

(7) $F^{0}(x)F_{\xi}(\nabla F^{0}(x))=x$.
\end{thm*}

%%%%

Let $\Omega$ be a  bounded domain. Here we provide the definition of Hausdorff dimension and singular set, see \cite{LY}.
\begin{defi}
 Let $A$ be a subset of $\mathbb{R}^{N}$, $0\leq s\leq\infty$ and $0\leq\delta\leq\infty$. Set
$$H_{\delta}^{s}:=\inf\left\{\sum_{j=1}^{\infty}\alpha(s)\left(\frac{diamC_{j}}{2}\right)^{s}\Big| A\subset\cup_{j=1}^{\infty}C_{j}, diamC_{j}\leq\delta\right\},$$
where $\alpha(s)=\frac{\pi^{\frac{s}{2}}}{\Gamma(\frac{s}{2}+1)}$, $0\leq s<\infty$ and $\Gamma(s)$ is the $\Gamma$-function. Let $H^{s}$ be the $s$-dimensional Hausdorff measure that is defined as
$$H^{s}(A):=\lim_{\delta\rightarrow0}H_{\delta}^{s}(A)=\sup_{\delta>0}H_{\delta}^{s}(A).$$
The Hausdorff dimension of a set $A\subset\mathbb{R}^{N}$ is defined as
$$H_{dim}(A):=\inf\{0\leq s<\infty| H^{s}(A)=0\}.$$
\end{defi}
Here is the definition of the singular set $S$, see \cite{W1}.
\begin{defi}
The singular set $S$ of a solution $u$ contain those point where in any neighborhood of this point $u$ is not bounded, its complement is the regular set of $u$.
\end{defi}
Here is our main result addressing partial regularity of solutions of (\ref{1.1e}).
\begin{thm}\label{thm5}
Assume that for all $x, y\in \mathbb{R}^{N}$,
\begin{equation}\label{1.7}
\langle F_{\xi}(x), F_{\xi}^{0}(y)\rangle=\frac{\langle x,y\rangle}{F(x)F^{0}(y)} . 
\end{equation}
 If $u$ is a stable weak solution of (\ref{1.1e}) with $\Omega\subset\mathbb{R}^{N}$, where $\Omega$ is a bounded domain, then the Hausdorff dimension of the singular set $S$ does not exceed $N-10$.
\end{thm}
When $F(\xi)=|\xi|$, for the Laplacian operator, Da Lio \cite{D} proved that the Hausdorff dimension of singular set of stable solution  is at most $1$ in dimension $N=3$.  Wang \cite{W1,W2} extended this result to higher dimensions, and showed that the Hausdorff dimension does not exceed $N-10$.

We now consider $\Omega$ to the entire space $\mathbb{R}^{N}$.  Here we list our main results for such domains. The first result is a Liouville  theorem for stable solutions. %The following is our main results.
\begin{thm}\label{thm1}
 If $N < 10$, then there is no stable solution of equation  (\ref{1.1e}).
\end{thm}

The following is the Liouville theorem for finite Morse index solution.

\begin{thm}\label{thm3}
For $3\leq N\leq9$, under the assumption of (\ref{1.7}), then equation (\ref{1.1e}) does not admit any solution stable outside a compact set of $\mathbb{R}^{N}$. If $N=2$, then
\begin{align}\label{1.6}
u(x)=-2\log(1+\frac{1}{8}\lambda^{2}F^{0}(x-x_{0})^{2})+2\log\lambda,
\end{align}
 for some $\lambda>0$ and $x_{0}\in\mathbb{R}^{2}$, is stable outside a compact set of $\mathbb{R}^{2}$ of  (\ref{1.1e}).
\end{thm}
When $F(\xi)=|\xi|$,   Farina in \cite{5} proved an analogues of Theorem \ref{thm1}. Dancer and Farina in \cite{3,3d} proved a counterpart of Theorem \ref{thm3}. The methods applied in here are the Moser iteration arguments developed in this context by Crandall and Rabinowitz \cite{CR}.

For power-type nonlinearities, we prove the following Liouville theorem for for stable solutions of (\ref{1.1}). This is a counterpart of Theorem \ref{thm1}.

\begin{thm}\label{thm1.5}
 The equation  (\ref{1.1}) does not admit positive weak stable solution if
\begin{enumerate}
\item[(i)] $f(u)=u^{p}$ for $p>3$ and $N<\frac{6p+4\sqrt{p^{2}-p}-2}{p-1}$.

\item[(ii)]  $f(u)=-u^{-p}$  for  $p>\frac{1}{3}$ and $N<\frac{6p+4\sqrt{p^{2}+p}+2}{p+1}$.
\end{enumerate}
\end{thm}

When $F(\xi)=|\xi|$, the above result is given by  Farina in \cite{F} and Esposito et al. in \cite{egg, zz} for Part (i)  and Part (ii), respectively.

This article is organized as  follows. In Section \ref{Sec-IE}, we first recall a well-known sharp anisotropic Hardy's inequality.  Then, we prove certain integral estimates using Moser iteration arguments. All of these inequalities are essential tools in next sections. In Section \ref{Sec-PR}, we prove the partial regularity result, i.e.,  Theorem \ref{thm5}. In Section \ref{Sec-LT}, we prove  Liouville theorems for stable solutions and for finite Moser index solutions, i.e.,  Theorem \ref{thm1}, Theorem \ref{thm3} and Theorem \ref{thm1.5}. In addition, we also show the existence of the finite Moser index solutions. In the last Section \ref{Sec-MF}, we discuss monotonicity formulas which are of independent interests in this context.

\section{ Integral Estimates }\label{Sec-IE}
In this section, we provide some essential elliptic estimates and inequalities needed to establish our main results. We start with the following sharp anisotropic Hardy inequality, given in \cite{MST}.

\begin{pro}\label{pro2}
Assume $1\leq s<N$ or $s>N$, let $\Omega$ be a domain in $\mathbb{R}^{N}$. Then the following inequality
$$\left|\frac{N-s}{s}\right|^{s}\int_{\Omega}\frac{|\varphi|^{s}}{(F^{0}(x))^{s}}dx\leq\int_{\Omega}\left|\frac{x}{F^{0}(x)}\cdot\nabla\varphi\right|^{s}dx$$
holds true for any $\varphi\in C_{c}^{\infty}(\Omega)$ if $1\leq s<N$, and for any $\varphi\in C_{c}^{\infty}(\Omega\setminus\{0\})$ if $s>N$.
\end{pro}
We now prove some integral estimates for stable solutions.  The methods and ideas are inspired by Moser iteration arguments given in \cite{5,F,CR} and references therein.
\begin{pro}\label{pro3}
Assume that $N\geq2$ and $\Omega$ is a domain (possibly unbounded) of $\mathbb{R}^{N}$. Let $u$ be a stable weak solution of (\ref{1.1}). 
\begin{enumerate}
\item[(i)] If $f(u)=e^{u}$ then for any integer $m\geq10$ and any $\alpha\in(0,4)$, we have
\begin{equation}\label{2.2}
\int_{\Omega}e^{(\alpha+1)u}\psi^{2m}dx\leq C\int_{\Omega}\left(|\nabla\psi|^{2}+|\nabla\psi|^{4}\right)^{\alpha+1}dx . 
\end{equation}

\item[(ii)] 
If $f(u)=u^{p}$ with $p>3$, then for any integer $m\geq\frac{4\sqrt{p^{2}-p}+6p-2}{p-3}$ and $p-\sqrt{p^{2}-p}<\alpha<p+\sqrt{p^{2}-p}$, we have
\begin{align}\label{2.3}
\int_{\Omega}u^{p+2\alpha-1}\psi^{2m}dx\leq C\int_{\Omega}\left(|\nabla\psi|^{\frac{2}{p-1}}+|\nabla\psi|^{\frac{4}{p-3}}\right)^{2\alpha+p-1}dx . 
\end{align}

\item [(iii)] 
If $f(u)=-u^{-p}$ with $p>\frac{1}{3}$, then for any integer $m\geq\max\{\frac{3p+2\sqrt{p^{2}+p}+1}{p+1},\frac{6p+4\sqrt{p^{2}+p}+2}{p+3}\}$ and $1<\alpha<p+\sqrt{p^{2}+p}$, we have
\begin{align}\label{2.5}
\int_{\Omega}u^{-2\alpha-p-1}\psi^{2m}dx\leq C\int_{\Omega}\left(|\nabla\psi|^{\frac{2}{p+1}}+|\nabla\psi|^{\frac{4}{p+3}}\right)^{2\alpha+p+1}dx .
\end{align}
\end{enumerate}
Here, $\psi$ is a test function $\psi\in C_{c}^{1}(\Omega)$ satisfying $0\leq\psi\leq1$ in $\Omega$.

\end{pro}

\begin{proof}
$(i)$ If $f(u)=e^{u}$, for any $\alpha\in(0,4)$ and any $k>0$, we set
$$ a_{k}(t)=\left\{
\begin{aligned}
e^{\frac{\alpha t}{2}}, &  & \indent\mbox{if}\indent t<k \\
[\frac{\alpha}{2}(t-k)+1]e^{\frac{\alpha k}{2}}, & & \indent\mbox{if}\indent t\geq k. \\
\end{aligned}
\right.
$$
and
$$ b_{k}(t)=\left\{
\begin{aligned}
e^{\alpha t}, &  & \indent\mbox{if}\indent t<k \\
[\alpha(t-k)+1]e^{\alpha k}, & & \indent\mbox{if}\indent t\geq k. \\
\end{aligned}
\right.
$$
Simple calculations yields
\begin{align}\label{2.6}
a_{k}^{2}(t)\geq b_{k}(t),\indent (a'_{k}(t))^{2}=\frac{\alpha}{4}b_{k}'(t) ,
\end{align}
and
\begin{align}\label{2.7}
(a_{k}'(t))^{-2}(a_{k}(t))^{4}\leq c_{1}e^{\alpha t},\indent (a_{k}(t))^{2}\leq e^{\alpha t}, \indent (b_{k}'(t))^{-1}(b_{k}(t))^{2}\leq c_{2}e^{\alpha t},
\end{align}
for some positive constant $c_{1}$ and $c_{2}$ which depends only on $\alpha$. For any $\phi\in C_{c}^{1}(\Omega)$, take $b_{k}(u)\phi^{2}$ as the test function, multiply (\ref{1.1}) and integrate by parts, it follows from Theorem \myref{A}, we have
\begin{align}
&\int_{\Omega}-div(F(\nabla u)F_{\xi}(\nabla u))b_{k}(u)\phi^{2}dx\nonumber\\
&=\int_{\Omega}-\frac{\partial}{\partial x_{i}}(F(\nabla u)F_{\xi_{i}}(\nabla u))b_{k}(u)\phi^{2}dx\nonumber\\
&=\int_{\Omega}F(\nabla u)F_{\xi_{i}}(\nabla u)b_{k}'(u)u_{x_{i}}\phi^{2}+F(\nabla u)F_{\xi_{i}}(\nabla u)b_{k}(u)2\phi\phi_{x_{i}}dx\nonumber\\
&=\int_{\Omega}F^{2}(\nabla u)b_{k}'(u)\phi^{2}+F(\nabla u)F_{\xi_{i}}(\nabla u)b_{k}(u)2\phi\phi_{x_{i}}dx\nonumber\\
&=\int_{\Omega}e^{u}b_{k}(u)\phi^{2}dx . \nonumber
\end{align}
It follows that
\begin{align}
\int_{\Omega}F^{2}(\nabla u)b_{k}'(u)\phi^{2}dx \leq 2C\int_{\Omega}F(\nabla u)b_{k}(u)|\phi||\nabla\phi|dx+\int_{\Omega}e^{u}b_{k}(u)\phi^{2}dx,\nonumber
\end{align}
by the Cauchy inequality, we have
 \begin{align}\label{2.8}
\int_{\Omega}F^{2}(\nabla u)b_{k}'(u)\phi^{2}dx \leq \frac{2C}{(1-2C\varepsilon)\varepsilon}\int_{\Omega}(b_{k}'(u))^{-1}b_{k}^{2}(u)|\nabla\phi|^{2}dx+\frac{1}{1-2C\varepsilon}\int_{\Omega}e^{u}b_{k}(u)\phi^{2}dx.
\end{align}
Since $u$ is stable solution of equation (\ref{1.1}), hence, for any $\varphi\in C_{c}^{1}(\Omega)$ we have
\begin{align}\label{+2.8}
\int_{\Omega}F_{\xi_{i}}(\nabla u)F_{\xi_{j}}(\nabla u)\varphi_{x_{i}}\varphi_{x_{j}}+F(\nabla u)F_{\xi_{i}\xi_{j}}(\nabla u)\varphi_{x_{i}}\varphi_{x_{j}}-e^{u}\varphi^{2}dx\geq0.
\end{align}\label{2.9}
Take $\varphi=a_{k}(u)\phi$, easy to check $\varphi_{x_{i}}=a_{k}'(u)u_{x_{i}}\phi+a_{k}(u)\phi_{x_{i}}$, using Theorem \myref{A} and Cauchy inequality, we have
\begin{align}\label{2.10}
&\int_{\Omega}F_{\xi_{i}}(\nabla u)F_{\xi_{j}}(\nabla u)\varphi_{x_{i}}\varphi_{x_{j}}dx\nonumber\\
&=\int_{\Omega}F_{\xi_{i}}(\nabla u)F_{\xi_{j}}(\nabla u)(a_{k}'(u)u_{x_{i}}\phi+a_{k}(u)\phi_{x_{i}})(a_{k}'(u)u_{x_{j}}\phi+a_{k}(u)\phi_{x_{j}})dx\nonumber\\
&\leq\int_{\Omega}(1+2C\varepsilon_{1})F^{2}(\nabla u)(a_{k}'(u))^{2}\phi^{2}+(C^{2}+\frac{2C}{\varepsilon_{1}})(a_{k}(u))^{2}|\nabla\phi|^{2}dx,
\end{align}
and
\begin{align}\label{2.11}
&\int_{\Omega}F(\nabla u)F_{\xi_{i}\xi_{j}}(\nabla u)\varphi_{x_{i}}\varphi_{x_{j}}dx\nonumber\\
&=\int_{\Omega}F(\nabla u)F_{\xi_{i}\xi_{j}}(\nabla u)(a_{k}'(u)u_{x_{i}}\phi+a_{k}(u)\phi_{x_{i}})(a_{k}'(u)u_{x_{j}}\phi+a_{k}(u)\phi_{x_{j}})dx\nonumber\\
&=\int_{\Omega}F(\nabla u)F_{\xi_{i}\xi_{j}}(a_{k}(u))^{2}\phi_{x_{i}}\phi_{x_{j}}dx\leq\Lambda\int_{\Omega}F(\nabla u)(a_{k}(u))^{2}|\nabla\phi|^{2}dx\nonumber\\
&\leq\Lambda\varepsilon_{2}\int_{\Omega}F^{2}(\nabla u)(a_{k}'(u))^{2}\phi^{2}dx+\frac{\Lambda}{\varepsilon_{2}}\int_{\Omega}(a_{k}'(u))^{-2}(a_{k}(u))^{4}\frac{|\nabla\phi|^{4}}{\phi^{2}}dx.
\end{align}
Combine (\ref{2.6}), (\ref{2.7}), (\ref{2.8}), (\ref{+2.8}), (\ref{2.10}) and (\ref{2.11}) we obtain
\begin{align}
\int_{\Omega}e^{u}(a_{k}(u))^{2}\phi^{2}dx\nonumber
&\leq\frac{\alpha(1+2C\varepsilon_{1}+\Lambda\varepsilon_{2})}{4(1-2C\varepsilon)}\int_{\Omega}e^{u}(a_{k}(u))^{2}\phi^{2}dx+\frac{\Lambda}{\varepsilon_{2}}c_{1}\int_{\Omega}e^{\alpha u}\frac{|\nabla\phi|^{4}}{\phi^{2}}dx\nonumber\\
&\indent+\left[\frac{2C\alpha(1+2C\varepsilon_{1}+\Lambda\varepsilon_{2})}{4(1-2C\varepsilon)\varepsilon}c_{2}+C^{2}+\frac{2C}{\varepsilon_{1}}\right]\int_{\Omega}e^{\alpha u}|\nabla\phi|^{2}dx . \nonumber
\end{align}
Since $\alpha\in(0,4)$, so we can choose $\varepsilon$, $\varepsilon_{1}$ and $\varepsilon_{2}$ small enough, such that
$$\frac{\alpha(1+2C\varepsilon_{1}+\Lambda\varepsilon_{2})}{4(1-2C\varepsilon)}<1 . $$
Hence, we have
$$\int_{\Omega}e^{u}(a_{k}(u))^{2}\phi^{2}dx\leq C_{1}\int_{\Omega}e^{\alpha u}\frac{|\nabla\phi|^{4}}{\phi^{2}}dx+C_{2}\int_{\Omega}e^{\alpha u}|\nabla\phi|^{2}dx,$$
where $C_{1}$ and $C_{2}$ are positive constants and independent of $k$. Then let $k\rightarrow+\infty$, by Fatou's lemma, we have
$$\int_{\Omega}e^{(\alpha+1)u}\phi^{2}dx\leq C_{1}\int_{\Omega}e^{\alpha u}\frac{|\nabla\phi|^{4}}{\phi^{2}}dx+C_{2}\int_{\Omega}e^{\alpha u}|\nabla\phi|^{2}dx . $$
Let $\phi=\psi^{m}$ and $0\leq\psi\leq1$, by young's inequality, we have
\begin{align}
&\int_{\Omega}e^{(\alpha+1)u}\phi^{2}dx=\int_{\Omega}e^{(\alpha+1)u}\psi^{2m}dx\nonumber\\
&\leq \widetilde{C_{1}}\varepsilon\int_{\Omega}e^{(\alpha+1)u}\psi^{2m}dx+\frac{\widetilde{C_{1}}}{\varepsilon}\int_{\Omega}\left(|\psi|^{2m-2-2m\frac{\alpha}{\alpha+1}}|\nabla\psi|^{2}\right)^{\alpha+1}dx\nonumber\\
&\indent+\widetilde{C_{2}}\varepsilon\int_{\Omega}e^{(\alpha+1)u}\psi^{2m}dx+\frac{\widetilde{C_{2}}}{\varepsilon}\int_{\Omega}\left(|\psi|^{2m-4-2m\frac{\alpha}{\alpha+1}}|\nabla\psi|^{4}\right)^{\alpha+1}dx . \nonumber
\end{align}
Since $m\geq10$, we have $2m-4-2m\frac{\alpha}{\alpha+1}\geq0$ and we can choose $\varepsilon$ small such that
$$\int_{\Omega}e^{(\alpha+1)u}\psi^{2m}dx\leq C\int_{\Omega}(|\nabla\psi|^{2}+|\nabla\psi|^{4})^{\alpha+1}dx . $$
This completes the proof of (\ref{2.2}).

$(ii)$ If $f(u)=u^{p}$, we use the same method, for any $\alpha\in(p-\sqrt{p^{2}-p}, p+\sqrt{p^{2}-p})$ and any $k>0$, set
$$ a_{k}(t)=\left\{
\begin{aligned}
t^{\alpha}, &  & \indent\mbox{if}\indent t<k \\
\left[\frac{\alpha^{2}}{(2\alpha-1)k}(t-k)+1\right]k^{\alpha}, & & \indent\mbox{if}\indent t\geq k, \\
\end{aligned}
\right.
$$
and
$$ b_{k}(t)=\left\{
\begin{aligned}
t^{2\alpha-1}, &  & \indent\mbox{if}\indent t<k \\
\left[\frac{\alpha^{2}}{(2\alpha-1)k}(t-k)+1\right]k^{2\alpha-1}, & & \indent\mbox{if}\indent t\geq k, \\
\end{aligned}
\right.
$$
easy to check
\begin{align}\label{2.12}
(a_{k}'(t))^{2}=\frac{\alpha^{2}}{2\alpha-1}b_{k}'(t) \indent\mbox{and}\indent (a_{k}(t))^{2}\geq tb_{k}(t),
\end{align}
and
\begin{align}\label{2.13}
(a_{k}'(t))^{-2}(a_{k}(t))^{4}\leq C_{3}t^{2\alpha+2},\indent (a_{k}(t))^{2}\leq C_{4}t^{2\alpha} \indent\mbox{and}\indent (b_{k}'(t))^{-1}(b_{k}(t))^{2}\leq C_{5}t^{2\alpha},
\end{align}
where $C_{3}$, $C_{4}$ and $C_{5}$ are positive constant and independent of $k$. For any $\phi\in C_{c}^{1}(\Omega)$, take $b_{k}(u)\phi^{2}$ as the test function, multiple (\ref{1.1}) and integrate by parts, it follows from Theorem \myref{A}, we have
\begin{align}\label{2.14}
&\int_{\Omega}F^{2}(\nabla u)b_{k}'(u)\phi^{2}dx\nonumber\\
&\leq\frac{2C}{(1-2C\varepsilon)\varepsilon}\int_{\Omega}(b_{k}'(u))^{-1}(b_{k}(u))^{2}|\nabla\phi|^{2}dx+\frac{1}{1-2C\varepsilon}\int_{\Omega}u^{p}b_{k}(u)\phi^{2}dx.
\end{align}
Since $u$ is a stable solution of equation (\ref{1.1}), hence for any $\varphi\in C_{c}^{1}(\Omega)$, holds
\begin{align}\label{+2.14}
\int_{\Omega}F_{\xi_{i}}(\nabla u)F_{\xi_{j}}(\nabla u)\varphi_{x_{i}}\varphi_{x_{j}}+F(\nabla u)F_{\xi_{i}\xi_{j}}(\nabla u)\varphi_{x_{i}}\varphi_{x_{j}}-pu^{p-1}\varphi^{2}dx\geq0,
\end{align}
take $\varphi=a_{k}(u)\phi$, easy to see $\varphi_{x_{i}}=a_{k}'(u)u_{x_{i}}\phi+a_{k}(u)\phi_{x_{i}}$, by Theorem \myref{A} and Cauchy inequality, we have
\begin{align}\label{+2.15}
&p\int_{\Omega}u^{p-1}(a_{k}(u))^{2}\phi^{2}dx\nonumber\\
&\leq(1+2C\varepsilon_{1}+\Lambda\varepsilon_{2})\int_{\Omega}F^{2}(\nabla u)(a_{k}'(u))^{2}\phi^{2}dx\nonumber\\
&\indent+(C^{2}+\frac{2C}{\varepsilon_{1}})\int_{\Omega}(a_{k}(u))^{2}|\nabla\phi|^{2}dx+\frac{\Lambda}{\varepsilon_{2}}\int_{\Omega}(a_{k}'(u))^{-2}(a_{k}(u))^{4}\frac{|\nabla\phi|^{4}}{\phi^{2}}dx . 
\end{align}
It follows from (\ref{2.12}), (\ref{2.13}), (\ref{2.14}), (\ref{+2.14}) and (\ref{+2.15}) that
\begin{align}
&p\int_{\Omega}u^{p-1}(a_{k}(u))^{2}\phi^{2}dx\nonumber\\
&\leq\frac{\alpha^{2}(1+2C\varepsilon_{1}+\Lambda\varepsilon_{2})}{(2\alpha-1)(1-2C\varepsilon)}\int_{\Omega}u^{p-1}(a_{k}(u))^{2}\phi^{2}dx+\frac{\Lambda}{\varepsilon_{2}}C_{3}\int_{\Omega}u^{2\alpha+2}\frac{|\nabla\phi|^{4}}{\phi^{2}}dx\nonumber\\
&\indent+\left[\frac{2C\alpha(1+2C\varepsilon_{1}+\Lambda\varepsilon_{2})}{(2\alpha-1)(1-2C\varepsilon)\varepsilon}C_{5}+(C ^{2}+\frac{2C}{\varepsilon_{1}})C_{4}\right]\int_{\Omega}u^{2\alpha}|\nabla\phi|^{2}dx . \nonumber
\end{align}
Since $p-\sqrt{p^{2}-p}<\alpha<p+\sqrt{p^{2}-p}$, we can choose $\varepsilon$, $\varepsilon_{1}$ and $\varepsilon_{2}$ small enough such that
$$p>\frac{\alpha^{2}(1+2C\varepsilon_{1}+\Lambda\varepsilon_{2})}{(2\alpha-1)(1-2C\varepsilon)} . $$
It follows that
$$\int_{\Omega}u^{p-1}(a_{k}(u))^{2}\phi^{2}dx\leq C_{6}\int_{\Omega}u^{2\alpha}|\nabla\phi|^{2}dx+C_{7}\int_{\Omega}u^{2\alpha+2}\frac{|\nabla\phi|^{4}}{\phi^{2}}dx,$$
where $C_{5}$, $C_{6}$ are positive constant and independent of $k$, so let $k\rightarrow+\infty$, by Fatou's lemma we have
$$\int_{\Omega}u^{2\alpha+p-1}\phi^{2}dx\leq C_{6}\int_{\Omega}u^{2\alpha}|\nabla\phi|^{2}dx+C_{7}\int_{\Omega}u^{2\alpha+2}\frac{|\nabla\phi|^{4}}{\phi^{2}}dx . $$
Since $p>3$, let $\phi=\psi^{m}$ and $0\leq\psi\leq1$, by Young's inequality we have
\begin{align}
&\int_{\Omega}u^{2\alpha+p-1}\phi^{2}dx=\int_{\Omega}u^{2\alpha+p-1}\psi^{2m}dx\nonumber\\
&\leq\widetilde{C}_{6}\varepsilon\int_{\Omega}u^{2\alpha+p-1}\psi^{2m}dx+\frac{\widetilde{C}_{6}}{\varepsilon}\int_{\Omega}\left[|\nabla\psi|^{2}\psi^{2m-2-2m\frac{2\alpha}{2\alpha+p-1}}\right]^{\frac{2\alpha+p-1}{p-1}}dx\nonumber\\
&\indent+\widetilde{C}_{7}\varepsilon\int_{\Omega}u^{2\alpha+p-1}\psi^{2m}dx+\frac{\widetilde{C}_{7}}{\varepsilon}\int_{\Omega}\left[|\nabla\psi|^{4}\psi^{2m-4-2m\frac{2\alpha+2}{2\alpha+p-1}}\right]^{\frac{2\alpha+p-1}{p-3}}dx . \nonumber
\end{align}
Since $m\geq\frac{6p+4\sqrt{p^{2}-p}-2}{p-3}$, we have
$$2m-4-2m\frac{2\alpha+2}{2\alpha+p-1}\geq0,$$
This finishes the proof of (\ref{2.3}).

$(iii)$ If $f(u)=-u^{-p}$, for any $\alpha\in(1, p+\sqrt{p^{2}+p})$ and for any $k>0$, we set
$$ a_{k}(t)=\left\{
\begin{aligned}
t^{-\alpha}, &  & \indent\mbox{if}\indent \frac{1}{t}<\frac{1}{k} \\
\left[\frac{\alpha^{2}}{(2\alpha+1)k}(k-t)+1\right]k^{-\alpha}, & & \indent\mbox{if}\indent \frac{1}{t}\geq \frac{1}{k}, \\
\end{aligned}
\right.
$$
and
$$ b_{k}(t)=\left\{
\begin{aligned}
t^{-(2\alpha+1)}, &  & \indent\mbox{if}\indent \frac{1}{t}<\frac{1}{k} \\
\left[\frac{\alpha^{2}}{(2\alpha+1)k}(k-t)+1\right]k^{-(2\alpha+1)}, & & \indent\mbox{if}\indent \frac{1}{t}\geq \frac{1}{k}, \\
\end{aligned}
\right.
$$
easy to check
\begin{align}\label{2.15}
(a_{k}'(t))^{2}=\frac{\alpha^{2}}{2\alpha+1}|b_{k}'(t)|\indent\mbox{and}\indent (a_{k}(t))^{2}\geq tb_{k}(t),
\end{align}
and
\begin{align}\label{2.16}
|a_{k}'(t)|^{-2}(a_{k}(t))^{4}\leq c_{3}t^{-2\alpha+2}, \indent (a_{k}(t))^{2}\leq c_{4}t^{-2\alpha}\indent\mbox{and}\indent |b_{k}'(t)|^{-1}(b_{k}(t))^{2}\leq c_{5}t^{-2\alpha},
\end{align}
where $c_{3}$, $c_{4}$ and $c_{5}$ are positive constant and independent of $k$.
For any $\phi\in C_{c}^{1}(\Omega)$, take $b_{k}(u)\phi^{2}$ as the test function, multiple (\ref{1.1}) and integrate by parts, it follows from Theorem \myref{A}, we have
\begin{align}\label{2.17}
\int_{\Omega}F^{2}(\nabla u)|b_{k}'(u)|\phi^{2}dx \leq\frac{2C}{(1-2C\varepsilon)\varepsilon}\int_{\Omega}|b_{k}'(u)|^{-1}(b_{k}(u))^{2}|\nabla\phi|^{2}dx+\int_{\Omega}\frac{1}{1-2C\varepsilon}\int_{\Omega}u^{-p}b_{k}(u)\phi^{2}dx.
\end{align}
Since $u$ is stable solution of equation (\ref{1.1}), hence for any $\varphi\in C_{c}^{1}(\Omega)$, we have
\begin{align}\label{+2.17}
\int_{\Omega}F_{\xi_{i}}(\nabla u)F_{\xi_{j}}(\nabla u)\varphi_{x_{i}}\varphi_{x_{j}}+F(\nabla u)F_{\xi_{i}\xi_{j}}(\nabla u)\varphi_{x_{i}}\varphi_{x_{j}}-pu^{-p-1}\varphi^{2}dx\geq0 . 
\end{align}
Take $\varphi=a_{k}(u)\phi$, easy to check $\varphi_{x_{i}}=a_{k}'(u)u_{x_{i}}\phi+a_{k}(u)\phi_{x_{i}}$, by the Theorem \myref{A} and Cauchy inequality, we have
\begin{align}\label{2.18}
&p\int_{\Omega}u^{-p-1}(a_{k}(u))^{2}\phi^{2}dx\nonumber\\
&\leq(1+2C\varepsilon_{1}+\Lambda\varepsilon_{2})\int_{\Omega}F^{2}(\nabla u)(a_{k}'(u))^{2}\phi^{2}dx\nonumber\\
&\indent+(C^{2}+\frac{2C}{\varepsilon_{1}})\int_{\Omega}(a_{k}(u))^{2}|\nabla\phi|^{2}dx+\frac{\Lambda}{\varepsilon_{2}}\int_{\Omega}(a_{k}'(u))^{-2}(a_{k}(u))^{4}\frac{|\nabla\phi|^{4}}{\phi^{2}}dx . 
\end{align}
It follows from (\ref{2.15}), (\ref{2.16}), (\ref{2.17}), (\ref{+2.17}) and (\ref{2.18}) that
\begin{align}
&p\int_{\Omega}u^{-p-1}(a_{k}(u))^{2}\phi^{2}dx\nonumber\\
&\leq\left[\frac{2C\alpha^{2}(1+2C\varepsilon_{1}+\Lambda\varepsilon_{2})}{(2\alpha+1)(1-2C\varepsilon)\varepsilon}c_{5}+(C^{2}+\frac{2C}{\varepsilon_{1}})c_{4}\right]\int_{\Omega}u^{-2\alpha}|\nabla\phi|^{2}dx\nonumber\\
&\indent+\frac{\alpha^{2}(1+2C\varepsilon_{1}+\Lambda\varepsilon_{2})}{(2\alpha+1)(1-2C\varepsilon)}\int_{\Omega}u^{-p-1}(a_{k}(u))^{2}\phi^{2}dx+\frac{\Lambda}{\varepsilon_{2}}c_{3}\int_{\Omega}u^{-2\alpha+2}\frac{|\nabla\phi|^{4}}{\phi^{2}}dx . \nonumber
\end{align}
Since $\alpha\in(1,p+\sqrt{p^{2}+p})$, so we can choose $\varepsilon$, $\varepsilon_{1}$ and $\varepsilon_{2}$ small enough such that
$$p>\frac{\alpha^{2}(1+2C\varepsilon_{1}+\Lambda\varepsilon_{2})}{(2\alpha+1)(1-2C\varepsilon)} . $$
Hence, we have
$$\int_{\Omega}u^{-p-1}(a_{k}(u))^{2}\phi^{2}dx
\leq c_{6}\int_{\Omega}u^{-2\alpha}|\nabla\phi|^{2}dx
+c_{7}\int_{\Omega}u^{-2\alpha+2}\frac{|\nabla\phi|^{4}}{\phi^{2}}dx,$$
where $c_{6}$, $c_{7}$ are positive constants and independent of $k$, let $k\rightarrow+\infty$, by Fatou's lemma we have
$$\int_{\Omega}u^{-2\alpha-p-1}\phi^{2}dx
\leq c_{6}\int_{\Omega}u^{-2\alpha}|\nabla\phi|^{2}dx
+c_{7}\int_{\Omega}u^{-2\alpha+2}\frac{|\nabla\phi|^{4}}{\phi^{2}}dx . $$
Let $\phi=\psi^{m}$ and $0\leq\psi\leq1$, by Young's inequality, we have
\begin{align}
\int_{\Omega}u^{-2\alpha-p-1}\psi^{2m}dx
&\leq \widetilde{c}_{6}\int_{\Omega}(\psi^{2m-2-2m\frac{2\alpha}{p+1+2\alpha}}|\nabla\psi|^{2})^{\frac{p+1+2\alpha}{p+1}}dx\nonumber\\
&\indent+\widetilde{c}_{7}\int_{\Omega}(\psi^{2m-4-2m\frac{2\alpha-2}{p+1+2\alpha}}|\nabla\psi|^{4})^{\frac{p+1+2\alpha}{p+3}}dx . \nonumber
\end{align}
Since $m\geq\max\{\frac{3p+2\sqrt{p^{2}+p}+1}{p+1},\frac{6p+4\sqrt{p^{2}+p}+2}{p+3}\}$, so we have
$$2m-2-2m\frac{2\alpha}{p+1+2\alpha}\geq0,$$
and
$$2m-4-2m\frac{2\alpha-2}{p+1+2\alpha}\geq0,$$
This completes the proof of  (\ref{2.5}).
\end{proof}

\section{ The Partial Regularity Result  }\label{Sec-PR}

To prove our regularity theorem, the level set method plays an important role, see \cite{RW,WX12}, in order to use this method, let us first recall the important tools: the co-area formula and isoperimetric inequality for anisotropic version.  We define the total variation of $u\in BV(\Omega)$ with respect to $F$ by
$$\int_{\Omega}|\nabla u|_{F}:=\sup\left\{\int_{\Omega}u \text{div} \sigma dx: \sigma\in C_{c}^{1}(\Omega; \mathbb{R}^{N}), F^{0}(\sigma)\leq1\right\}.$$
From this definition, the perimeter of $E\subset\Omega$ is defined as
$$P_{F}(E):=\int_{\Omega}|\nabla\chi_{E}|_{F} , $$
where $\chi_{E}$ is the characteristic function of $E$. Then the co-area formula
\begin{align}
\int_{\Omega}|\nabla u|_{F}=\int_{0}^{\infty}P_{F}(\{|u|>t\})dt,\nonumber
\end{align}
and the isoperimetric inequality
\begin{align}
P_{F}(E)\geq N\kappa_{0}^{\frac{1}{N}}|E|^{1-\frac{1}{N}} , \nonumber
\end{align}
hold, and the equality holds if and only if $E$ is a Wulff ball, for the proof we refer to \cite{AFTL, FM}. Moreover, in \cite{AB}, we know that if $u\in W^{1, 1}(\Omega)$, then
$$\int_{\Omega}|\nabla u|_{F}=\int_{\Omega}F(\nabla u)dx , $$
and the co-area formula becomes
$$-\frac{d}{dt}\int_{\{u>t\}}F(\nabla u)dx=P_{F}(\{u>t\}) , $$
for almost every $t$.

Here we recall the definition of the Morrey space $M^{p}(\Omega)$, see \cite{GT},
\begin{defi}
A function $f\in L^{1}(\Omega)$ is said to belong to $M^{p}(\Omega)$, $1\leq p\leq\infty$, if there exists a constant $K$ such that
$$\int_{\Omega\cap B_{r}}|f|\leq Kr^{N(1-\frac{1}{p})},$$
for all $B_{r}\subset\mathbb{R}^{N}$, with the norm
$$\parallel f\parallel_{M^{p}(\Omega)}=\inf\{K| \int_{\Omega\cap B_{r}}|f|\leq Kr^{N(1-\frac{1}{p})}\}.$$
\end{defi}

In order to prove the main result, we will need the following decay estimate of equation (\ref{1.1e}). Without loss of generality, we always assume $\Omega=B_{2}(0)$ in (\ref{1.1e}).

\begin{lem}\label{lem4.1}
Under the assumption of (\ref{1.7}), there exist $\varepsilon_{0}>0$, $r\in(0, \frac{1}{2})$, which depend only on the dimension $N$, such that for a stable solution $u$ of (\ref{1.1e}), if
$$2^{2-N}\int_{B_{2}(0)}e^{u}dx\leq\varepsilon,$$
where $\varepsilon\leq\varepsilon_{0}$, then
\begin{align}\label{3.6}
r^{2-N}\int_{B_{r}(0)}e^{u}dx\leq\frac{1}{2}\varepsilon.
\end{align}
\end{lem}

\begin{proof}
By Proposition \ref{pro3}, let $\alpha=1$, take $\Omega=B_{2}(0)$ and $\psi=1$ in $B_{1}(0)$, we have
$$\int_{B_{1}(0)}e^{2u}dx\leq\int_{B_{2}(0)}e^{2u}\psi^{2m}dx\leq C\int_{B_{2}(0)}e^{u}(|\nabla\psi|^{2}+|\nabla\psi|^{4})dx,$$
hence $\parallel e^{u}\parallel_{L^{2}(B_{1}(0))}\leq C\varepsilon^{\frac{1}{2}}$. Take the decomposition $u=v+w$ in $B_{1}(0)$, where \begin{equation}\label{3.7}
\left\{
\begin{aligned}
-Q w&=0 \indent\mbox{in}\indent B_{1}(0)\\
w&=u \indent\mbox{on}\indent \partial B_{1}(0),\\
\end{aligned}
\right.
\end{equation}
and
\begin{equation}\label{3.8}
\left\{
\begin{aligned}
-\widetilde{Q} v:=-(Qu-Qw)&=e^{u} \indent\mbox{in}\indent B_{1}(0)\\
v&=0 \indent\mbox{on}\indent \partial B_{1}(0) . \\
\end{aligned}
\right.
\end{equation}
Set $\Omega=B_{1}(0)$, $\Omega_{t}=\{x\in\Omega|v>t\}$ and $\mu(t)=|\Omega_{t}|$, we have
\begin{align}
\int_{\Omega_{t}}e^{u}dx&=\int_{\Omega_{t}}-(Qu-Qw)dx\nonumber\\
&=\int_{\partial\Omega_{t}}\langle F(\nabla u)F_{\xi}(\nabla u)-F(\nabla w)F_{\xi}(\nabla w),\frac{\nabla(u-w)}{|\nabla(u-w)|}\rangle dS\nonumber\\
&\geq d_{0}\int_{\partial\Omega_{t}}\frac{F^{2}(\nabla(u-w))}{|\nabla(u-w)|}dS=d_{0}\int_{\partial\Omega_{t}}\frac{F^{2}(\nabla v)}{|\nabla v|}dS,\nonumber
\end{align}
where
$$d_{0}=\inf\left\{d_{X,Y}\Big| X,Y\in\mathbb{R}^{N}, X\neq0, Y\neq0, X\neq Y\right\} , $$
 with
 $$d_{X,Y}:=\frac{\langle F(X)F_{\xi}(X)-F(Y)F_{\xi}(Y), X-Y\rangle}{F^{2}(X-Y)} . $$
 It is straightforward to check $\min\{\frac{\lambda_{1}}{b^{2}}, 1\}\leq d_{0}\leq1$, where $\lambda_{1}$ is the smallest eigenvalue of $Hess(F^{2})$. By the isoperimetric inequality, the co-area formula and the H\"{o}lder inequality, we have
\begin{align}
N\kappa_{0}^{\frac{1}{N}}\mu(t)^{1-\frac{1}{N}}&\leq P_{F}(\{v>t\})=-\frac{d}{dt}\int_{\Omega_{t}}F(\nabla v)dx=\int_{\partial\Omega_{t}}\frac{F(\nabla v)}{|\nabla v|}dS\nonumber\\
&\leq\left(\int_{\partial\Omega_{t}}\frac{F^{2}(\nabla v)}{|\nabla v|}dS\right)^{\frac{1}{2}}\left(\int_{\partial\Omega_{t}}\frac{ 1}{|\nabla v|}dS\right)^{\frac{1}{2}}\nonumber\\
&\leq\left(\frac{1}{d_{0}}\int_{\Omega_{t}}e^{u}dx\right)^{\frac{1}{2}}\left(-\mu'(t)\right)^{\frac{1}{2}} . \nonumber
\end{align}
It follows that
$$-\mu'(t)\geq\frac{d_{0}N^{2}\kappa_{N}^{\frac{2}{N}}\mu(t)^{2-\frac{2}{N}}}{\int_{\Omega_{t}}e^{u}dx},$$
and hence
$$-\frac{dt}{d\mu}\leq\frac{\int_{\Omega_{t}}e^{u}dx}{d_{0}N^{2}\kappa_{N}^{\frac{2}{N}}\mu(t)^{2-\frac{2}{N}}}\leq C\frac{\int_{\Omega}e^{u}dx}{\mu(t)^{2-\frac{2}{N}}}.$$
Integrating the above inequality over $(\mu, |\Omega|)$, we have
\begin{align}
t(\mu)&\leq C\parallel e^{u}\parallel_{L^{1}(\Omega)}\int_{\mu}^{|\Omega|}\frac{1}{s^{2-\frac{2}{N}}}ds\nonumber\\
&\leq C\parallel e^{u}\parallel_{L^{1}(\Omega)}\left(\frac{1}{\mu^{1-\frac{2}{N}}}-\frac{1}{|\Omega|^{1-\frac{2}{N}}}\right) . \nonumber
\end{align}
Using the co-area formula again, we have
\begin{align}
\int_{\Omega}vdx&=\int_{0}^{\infty}t\cdot(-\mu'(t))dt=\int_{0}^{|\Omega|}t(\mu)d\mu\nonumber\\
&\leq\int_{0}^{|\Omega|}C\parallel e^{u}\parallel_{L^{1}(\Omega)}\left(\frac{1}{\mu^{1-\frac{2}{N}}}-\frac{1}{|\Omega|^{1-\frac{2}{N}}}\right)d\mu\nonumber\\
&\leq C|\Omega|^{\frac{2}{N}}\parallel e^{u}\parallel_{L^{1}(\Omega)}\leq   C\varepsilon.\nonumber
\end{align}
Hence, we have $\parallel v\parallel_{L^{1}(B_{1}(0))}\leq C\varepsilon$ and $\parallel e^{u}\parallel_{L^{2}(B_{1}(0))}\leq C\varepsilon^{\frac{1}{2}}$, by the elliptic estimate, we have $\parallel v\parallel_{W^{2,2}(B_{1}(0))}\leq C\varepsilon^{\frac{1}{2}}$, then it follows from the Sobolev embedding theorem, we get $\parallel v\parallel_{L^{\frac{2N}{N-4}}(B_{1}(0))}\leq C\varepsilon^{\frac{1}{2}}$, by interpolation inequality between $L^{q}$ space, we have
$$\parallel v\parallel_{L^{2}(B_{1}(0))}=\parallel v\parallel_{L^{1}(B_{1}(0))}^{\frac{4}{N+4}}\parallel v\parallel_{L^{\frac{2N}{N-4}}(B_{1}(0))}^{\frac{4}{N+4}}\leq C\varepsilon^{\alpha},$$
where $\alpha=\frac{N+8}{2N+8}>\frac{1}{2}$, then by interpolation inequality between Sobolev space, we get
$$\parallel \nabla v\parallel_{L^{2}(B_{1}(0))}\leq C\left(\varepsilon^{\frac{1}{4}(\alpha-\frac{1}{2})}\parallel \nabla^{2} v\parallel_{L^{2}(B_{1}(0))}+\varepsilon^{-\frac{1}{4}(\alpha-\frac{1}{2})}\parallel v\parallel_{L^{2}(B_{1}(0))}\right)\leq C\varepsilon^{\beta},$$
where $\beta>\frac{1}{2}$ depends only on $N$, it follows from the above inequality, we get
\begin{align}
\int_{B_{1}(0)}ve^{u}dx&=\int_{B_{1}(0)}-v\widetilde{Q}vdx
\leq d_{0}\int_{B_{1}(0)}F^{2}(\nabla v)dx\nonumber\\
&\leq d_{0}b^{2}\int_{B_{1}(0)}|\nabla v|^{2}dx\leq C\varepsilon^{2\beta}.\nonumber
\end{align}
We decompose the estimate of $r^{2-N}\int_{B_{r}(0)}e^{u}dx$ into two parts: $\{v\leq\varepsilon^{\gamma}\}$ and $\{v>\varepsilon^{\gamma}\}$, where $\gamma=\frac{1}{2}(2\beta-1)>0$. Since $Qw=0$, we have
$$Q(e^{w})=e^{w}F^{2}(\nabla w)\geq0,$$
under the assumption of (\ref{1.7}), we have the mean-value inequality, see \cite{6},
$$e^{w(y)}\leq \frac{1}{\kappa_{0}r^{N}}\int_{B_{r}(y)}e^{w(x)}dx,$$
for all $B_{r}(y)\subset B_{1}(0)$. For $r\in(0, \frac{1}{2})$, for any $x\in B_{r}(0)$ we have $B_{\frac{1}{2}}(x)\subset B_{1}(0)$, hence,
$$r^{-N}\int_{B_{r}(0)}e^{w}dx\leq2^{N}\int_{B_{1}(0)}e^{w}dx\leq2^{N}\int_{B_{1}(0)}e^{u}dx . $$
It follows that
\begin{align}
r^{2-N}\int_{B_{r}(0)\cap\{v\leq\varepsilon^{\gamma}\}}e^{u}dx&\leq r^{2-N}\int_{B_{r}(0)\cap\{v\leq\varepsilon^{\gamma}\}}e^{\varepsilon^{\gamma}}e^{w}dx\nonumber\\
&\leq r^{2}e^{\varepsilon^{\gamma}}r^{-N}\int_{B_{r}(0)}e^{w}dx\nonumber\\
&\leq2^{N} r^{2}e^{\varepsilon^{\gamma}}\int_{B_{1}(0)}e^{u}dx\leq Cr^{2}\varepsilon.\nonumber
\end{align}
For the second part,
\begin{align}
r^{2-N}\int_{B_{r}(0)\cap\{v>\varepsilon^{\gamma}\}}e^{u}dx&\leq r^{2-N}\int_{B_{r}(0)\cap\{v\leq\varepsilon^{\gamma}\}}\frac{v}{\varepsilon^{\gamma}}e^{u}dx\nonumber\\
&\leq \varepsilon^{-\gamma}r^{2-N}\int_{B_{r}(0)}ve^{u}dx\nonumber\\
&\leq Cr^{2-N}\varepsilon^{2\beta-\gamma} . \nonumber
\end{align}
Hence, we have
$$r^{2-N}\int_{B_{r}(0)}e^{u}dx\leq Cr^{2}\varepsilon+Cr^{2-N}\varepsilon^{2\beta-\gamma},$$
note that $2\beta-\gamma>1$, we can choose $r$ small enough, then $\varepsilon_{0}$ small enough, such that for any $\varepsilon\leq\varepsilon_{0}$, holds
$$Cr^{2}\varepsilon+Cr^{2-N}\varepsilon^{2\beta-\gamma}\leq\frac{1}{2}\varepsilon,$$
we obtain the conclusion.
\end{proof}

It follows from the above decay estimate Lemma \ref{lem4.1} and the priori estimate in Morrey space, we have the following $\varepsilon$-regularity theorem.

\begin{lem}\label{lem4.2}
Under the assumption of (\ref{1.7}), suppose $u$ is a stable weak solution of (\ref{1.1e}), if there exist $\varepsilon_{0}>0$ such that
$$\int_{B_{1}(0)}e^{u}dx\leq\varepsilon, $$
where $\varepsilon\leq\varepsilon_{0}$, then
$$\sup_{B_{\frac{1}{4}}(0)}u<\infty.$$
\end{lem}

\begin{proof}
We choose $\varepsilon_{0}$ small, such that for any $y\in B_{\frac{1}{2}}(0)$, holds
$$2^{N-2}\int_{B_{\frac{1}{2}}(y)}e^{u}\leq\varepsilon,$$
thus we can use Lemma \ref{lem4.1}, by a standard induction, we get $\exists \delta>0$ and $r<\frac{1}{2}$ such that $\forall y\in B_{\frac{1}{2}}(0)$, holds
$$\int_{B_{r}(y)}e^{u}dx\leq Cr^{N-2+\delta},$$
this implies $e^{u}\in M^{\frac{N}{2-\delta}}$. Take the decomposition $u=v+w$, where \begin{equation}\label{3.9}
\left\{
\begin{aligned}
-Q w&=0 \indent\mbox{in}\indent B_{\frac{1}{2}}(0)\\
w&=u \indent\mbox{on}\indent \partial B_{\frac{1}{2}}(0),\\
\end{aligned}
\right.
\end{equation}
and
\begin{equation}\label{3.10}
\left\{
\begin{aligned}
-\widetilde{Q} v:=-(Qu-Qw)&=e^{u} \indent\mbox{in}\indent B_{\frac{1}{2}}(0)\\
v&=0 \indent\mbox{on}\indent \partial B_{\frac{1}{2}}(0),\\
\end{aligned}
\right.
\end{equation}
from the elliptic estimate, we can get $w$ is bounded in $B_{\frac{1}{4}}(0)$. Next, to estimate $v$, we also use the level set method. Denote $\Omega=B_{\frac{1}{2}}(0)$, set $\Omega_{t}=\{x\in\Omega|v>t\}$ and $\mu(t)=|\Omega_{t}|$, we have
\begin{align}
\int_{\Omega_{t}}e^{u}dx=\int_{\Omega_{t}}-Qvdx=\int_{\partial\Omega_{t}}F(\nabla v)F_{\xi}(\nabla v)\frac{\nabla v}{|\nabla v|}dS=\int_{\partial\Omega_{t}}\frac{F^{2}(\nabla v)}{|\nabla v|}dS,\nonumber
\end{align}
by the isoperimetric inequality, the co-area formula and Holder's inequality, we have
\begin{align}
N\kappa_{0}^{1/N}\mu(t)^{1-1/N}&\leq P_{F}(\Omega_{t})=-\frac{d}{dt}\int_{\Omega_{t}}F(\nabla v)dx\nonumber\\
&=\int_{\partial\Omega_{t}}\frac{F(\nabla v)}{|\nabla v|}dS\leq\left(\int_{\partial\Omega_{t}}\frac{F^{2}(\nabla v)}{|\nabla v|}dS\right)^{1/2}\left(\int_{\partial\Omega_{t}}\frac{1}{|\nabla v|}dS\right)^{1/2}\nonumber\\
&=\left(\int_{\Omega_{t}}e^{u}dx\right)^{1/2}\left(-\mu'(t)\right)^{1/2} . \nonumber
\end{align}
It follows that
$$-\mu'(t)\geq\frac{N^{2}\kappa_{0}^{2/N}\mu_{t}^{2-2/N}}{\int_{\Omega_{t}}e^{u}dx}. $$
Hence
$$-\frac{dt}{d\mu}\leq\frac{\int_{\Omega_{t}}e^{u}dx}{N^{2}\kappa_{0}^{2/N}\mu(t)^{2-2/N}}\leq C\frac{\mu^{1-\frac{1}{\frac{N}{2-\delta}}}\parallel e^{u}\parallel_{M^{\frac{N}{2-\delta}}(\Omega_{t})}}{\mu^{2-2/N}}\leq C\frac{\parallel e^{u}\parallel_{M^{\frac{N}{2-\delta}}(\Omega)}}{\mu^{1-\frac{\delta}{N}}} . $$
Integrating the above inequality over $(\mu, |\Omega|)$, we have
\begin{align}
t(\mu)&\leq\int_{\mu}^{|\Omega|}C\frac{\parallel e^{u}\parallel_{M^{\frac{N}{2-\delta}}(\Omega)}}{s^{1-\frac{\delta}{N}}}ds\nonumber\\
&\leq C\parallel e^{u}\parallel_{M^{\frac{N}{2-\delta}}(\Omega)}\left(|\Omega|^{\frac{\delta}{N}}-\mu^{\frac{\delta}{N}}\right)<\infty. \nonumber
\end{align}
This inequality implies that $\parallel v\parallel_{L^{\infty}(\Omega)}<\infty$. Thus we get $u$ is bounded in $B_{\frac{1}{4}}(0)$.
\end{proof}

%\hspace*{\fill} \\

Now, we prove our main partial regularity result.

\begin{proof}[Proof of Theorem \ref{thm5}]
 The equation (\ref{1.1e}) is invariant under the rescaling
$$u^{r}(x)=u(rx)+2\log r,$$
from (\ref{2.2}), we know that $\forall p\in(1, 5)$, $\exists C>0$ such that
$$\int_{B_{r}(x)}e^{pu}\leq Cr^{N-2p}.$$
Hence, if
$$r^{2p-N}\int_{B_{r}(x)}e^{pu}\leq\varepsilon,$$
by H\"{o}lder's inequality, we have
$$\int_{B_{1}(x)}e^{u^{r}(y)}dy\leq\varepsilon . $$
Therefore, it follows from Lemma \ref{lem4.2}, we have $u^{r}(y)$ is bounded in $B_{\frac{1}{4}}(x)$, this implies that $u$ is bounded in $B_{\frac{r}{4}}(x)$. Thus, for any $x\in S$ and $r>0$, we have
$$r^{2p-N}\int_{B_{r}(x)}e^{pu}>\varepsilon . $$
From the Besicovitch covering Lemma, see \cite{LY}, we have
$$H^{N-2p}(S)=0. $$
Since $p$ is arbitrary in $(1, 5)$, we complete the proof.
\end{proof}

\section{ The Liouville Theorems }\label{Sec-LT}
In this section, we are mainly devoted to the proof of Liouville theorem for stable solutions and finite Morse index solutions.

\begin{proof}[Proof of Theorem \ref{1.1}]
By contradiction, suppose $u$ is the stable solution of equation (\ref{1.1e}). Proposition \ref{pro3} tell us that we can fix an integer $m\geq10$, and choose $\alpha\in(0,4)$ such that $N-2(\alpha+1)<0$, for every $x\in \mathbb{R}^{N}$ consider the function $\phi_{R}(x)=\phi(\frac{F^{0}(x)}{R})$, where $\phi\in C_{c}^{1}(\mathbb{R})$ satisfy $0\leq\phi\leq1$ everywhere on $\mathbb{R}$ and

$$ \phi(t)=\left\{
\begin{aligned}
1, &  & \indent\mbox{if}\indent |t|\leq1 \\
0, & & \indent\mbox{if}\indent |t|\geq2. \\
\end{aligned}
\right.
$$
For every $R>0$, we have
$$\int_{B_{R}(0)}e^{(\alpha+1)u}dx\leq \widetilde{C}R^{N-2(\alpha+1)},$$
where $\widetilde{C}$ is a positive constant independent on $R$. Letting $R\rightarrow+\infty$, we obtain $\int_{\mathbb{R}^{N}}e^{(\alpha+1)u}dx=0$, a contradiction.
\end{proof}

%\hspace*{\fill} \\

\begin{proof}[Proof of Theorem \ref{thm1.5}] Similarly, we can also fix $m$ and choose $\alpha$ such that
$$\int_{\mathbb{R}^{N}}u^{p+2\alpha-1}dx=0$$
and
$$\int_{\mathbb{R}^{N}}u^{-p-2\alpha-1}dx=0,$$
we obtain the contradiction.
\end{proof}

\begin{proof}[Proof of Theorem \ref{thm3}]
When $N=2$, without loss of generality, take $x_{0}=0$. We observe that there exists $R=R(\lambda)>1$ such that
$$e^{u(x)}\leq \frac{1}{4F^{0}(x)^{2}\ln^{2}(F^{0}(x))},$$
for $F^{0}(x)>R$. It is straightforward to see that  $v(x)=\ln^{\frac{1}{2}}(F^{0}(x))$ solves quasi-linear equation
$$-Qv=\frac{1}{4F^{0}(x)^{2}\ln^{2}(F^{0}(x))}v . $$
For any $\phi\in C_{c}^{\infty}(\mathbb{R}^{2}\setminus B_{R})$, we have
$$\int_{\mathbb{R}^{2}\setminus B_{R}}|F_{\xi}(\nabla v)\cdot\nabla\phi|^{2}-\frac{1}{4F^{0}(x)^{2}\ln^{2}(F^{0}(x))}\phi^{2}dx\geq0. $$
From the properties of $F$,  once can see that  $F_{\xi}(\nabla v)=-F_{\xi}(\nabla u)$. So,  we have
$$\int_{\mathbb{R}^{2}\setminus B_{R}}|F_{\xi}(\nabla u)\cdot\nabla\phi|^{2}+\sum_{i,j=1}^{2}F(\nabla u)F_{\xi_{i}\xi_{j}}(\nabla u)\phi_{x_{i}}\phi_{x_{j}}-e^{u}\phi^{2}dx\geq0 , $$
where we used the definition of stable solution. 

We now prove nonexistence of stable outside a compact set solutions of $\mathbb{R}^{N}$  when  $3\leq N\leq9$. By contradiction, we assume $u$ is a solution of (\ref{1.1e}) which is stable outside a compact set of $\mathbb{R}^{N}$. In order to get the contradiction, we will split it into four steps. 

\emph{Step 1}. There exists $R_{0}=R_{0}(u)>0$ such that

\emph{(a) for any $\alpha\in(0,4)$ and $r>R_{0}+3$ there exist positive constant $A$ and $B$ depending on $\alpha$, $N$ and $R_{0}$ but not $r$, holds}

\emph{\begin{align}
\label{3.2}
\int_{B_{r}\setminus B_{R_{0}+2}}e^{(\alpha+1)u}dx\leq A+Br^{N-2(\alpha+1)}.
\end{align} }

\emph{(b) For any $B_{2R}(y)\subset\{x\in\mathbb{R}^{N}: F^{0}(x)>R_{0}\}$ and $\alpha\in(0,4)$, we have}

\emph{\begin{align}
\label{3.3}
\int_{B_{2R}(y)}e^{(\alpha+1)u}dx\leq CR^{N-2(\alpha+1)},
\end{align}}

\emph{ where $C$ is a positive constant depending on $\alpha$, $N$ and $R_{0}$ but not on $R$ and $y$.}

Since $u$ is stable outside a compact set of $\mathbb{R}^{N}$, there exist $R_{0}>0$ such that proposition holds true with $\Omega:=\mathbb{R}^{N}\setminus\overline{B_{R_{0}}(0)}$, we fix $m=10$, and for every $r>R_{0}+3$, we consider the following test function $\xi_{r}\in C_{c}^{1}(\mathbb{R}^{N})$
$$ \xi_{r}(x)=\left\{
\begin{aligned}
\theta_{R_{0}}(F^{0}(x)), &  & \indent\mbox{if}\indent x\in B_{R_{0}+3} \\
\phi\left(\frac{F^{0}(x)}{r}\right), &  &\indent\mbox{if}\indent x\in \mathbb{R}^{N}\setminus B_{R_{0}+3}, \\
\end{aligned}
\right.
$$
where $\phi$ is defined in the Proof of Theorem \ref{thm1} and for $s>0$, $\theta_{s}$ satisfying $\theta_{s}\in C_{c}^{1}(\mathbb{R})$, $0\leq\theta_{s}\leq1$ everywhere on $\mathbb{R}$ and
$$ \theta_{s}(t)=\left\{
\begin{aligned}
0, &  & \indent\mbox{if}\indent |t|\leq s+1 \\
1, &  &\indent\mbox{if}\indent |t|\geq s+2. \\
\end{aligned}
\right.
$$
It follows from Proposition \ref{pro3} that
\begin{align}
\int_{B_{r}\setminus B_{R_{0}+2}}e^{(\alpha+1)u}dx&\leq\int_{\Omega}e^{(\alpha+1)u}dx\nonumber\\
&\leq C\int_{\Omega}\left(|\nabla\xi_{r}|^{2}+|\nabla\xi_{r}|^{4}\right)^{\alpha+1}dx\nonumber\\
&\leq C_{1}(\alpha,N,\theta_{R_{0}})+C_{2}(\alpha,N,\phi)r^{N-2(\alpha+1)},\nonumber
\end{align}
hence the inequality (\ref{3.2}) holds.

The integral estimate (\ref{3.3}) is obtained in the same way by using the test functions $\psi_{R,y}(x)=\phi(\frac{F^{0}(x-y)}{R})$ in Proposition \ref{pro3}.

\emph{Step 2}. There exist $\eta>0$ and $R_{1}=R_{1}(N,\eta,u)>R_{0}$ such that 
\begin{align}
\label{3.4}
\int_{\mathbb{R}^{N}\setminus B_{R_{1}}}e^{\frac{N}{2}u}dx\leq\eta^{\frac{N}{2}}.
\end{align}

Let $\alpha_{1}:=\frac{N-2}{2}\in(0,4)$, for $r>R_{0}+3$, by (\ref{3.2}) we have
$$\int_{B_{r}\setminus B_{R_{0}+2}}e^{\frac{N}{2}u}dx\leq\int_{B_{r}\setminus B_{R_{0}+2}}e^{(\alpha_{1}+1)u}dx\leq A+Br^{N-2(\alpha_{1}+1)},$$
let $r\rightarrow\infty$, then we obtain the result.

\emph{Step 3. } The following asymptotic limit holds,  $$\lim_{F^{0}(x)\rightarrow\infty}F^{0}(x)^{2}e^{u(x)}=0.$$
Set $\varepsilon=\frac{1}{10}$, we observe that $\frac{N}{2-\varepsilon}\in(1,5)$.  Since $3\leq N\leq9$,  there exist $\alpha_{2}=\alpha_{2}(N)\in(0,4)$ such that $\alpha_{2}+1=\frac{N}{2-\varepsilon}$. Next we fix $\eta>0$ and observe that $w=e^{u}$ satisfies
 $$-Q w-e^{u}w\leq0 \indent\mbox{in}\indent B_{2R}(y).$$
  Applying Harnack's inequality, see \cite{11}, for positive solutions of the quasi-linear equation
$$-Q w=e^{u}w,$$
we have, for any $t>1$
\begin{align}
\label{3.5}
\parallel w\parallel_{L^{\infty}(B_{R}(y))}\leq CR^{-\frac{N}{t}}\parallel w\parallel_{L^{t}(B_{2R}(y))},
\end{align}
where $C$ is a positive constant depending on $N$ and $R^{\varepsilon}\parallel e^{u}\parallel_{L^{\frac{N}{2-\varepsilon}}(B_{2R}(y))}$. In order to apply the above result, we consider point $y\in\mathbb{R}^{N}$ such that $F^{0}(y)>10R_{1}$ and set $R=\frac{F^{0}(y)}{4}$, $t=\frac{N}{2}>1$, hence $R_{1}>R_{0}$ is defined by step 2. this choose yields
$$B_{2R}(y)\subset\{x\in\mathbb{R}^{N}: F^{0}(x)>R_{0}\},$$
$$\int_{\{F^{0}(x)>R_{1}\}}e^{\frac{N}{2}u}dx<\eta^{\frac{N}{2}},$$
and
\begin{align}
R^{\varepsilon}\parallel e^{u}\parallel_{L^{\frac{N}{2-\varepsilon}}(B_{2R}(y))}&=R^{\varepsilon}\left(\int_{B_{2R}(y)}e^{(\alpha_{2}+1)u}dx\right)^{\frac{2-\varepsilon}{N}}\nonumber\leq R^{\varepsilon}\left[CR^{N-2(\alpha_{2}+1)}\right]^{\frac{2-\varepsilon}{N}}\leq C_{1}.\nonumber
\end{align}

\emph{Step 4.}  In this step, we complete the proof. Let $v(r)=\frac{1}{N\kappa_{0}r^{N-1}}\int_{\partial B_{r}}udS$, then
$$v'(r)=\frac{1}{N\kappa_{0}r^{N-1}}\int_{\partial B_{r}}\langle \nabla u,\frac{x}{r}\rangle dS,$$
by the assumption of (\ref{1.7}), we have $\langle \nabla u,x\rangle=F(\nabla u)\langle F_{\xi}(\nabla u),F_{\xi}^{0}(x)\rangle F^{0}(x)$ and $F^{0}(x)=r$, $\nu=F^{0}_{\xi}(x)$ on $\partial B_{r}$, integration by parts
\begin{align}
v'(r)=\frac{1}{N\kappa_{0}r^{N-1}}\int_{\partial B_{r}}\sum_{i=1}^{N}F(\nabla u)F_{\xi_{i}}(\nabla u)\nu_{i} dS=\frac{1}{N\kappa_{0}r^{N-1}}\int_{B_{r}}Qudx . \nonumber
\end{align}
Then
\begin{align}
-v'(r)&=\frac{1}{N\kappa_{0}r^{N-1}}\int_{B_{r}}-Qudx=\frac{1}{N\kappa_{0}r^{N-1}}\int_{B_{r}}e^{u}dx\nonumber\\
&\leq\frac{1}{N\kappa_{0}r^{N-1}}\left(\int_{B_{r}}e^{(\alpha+1)u}dx\right)^{\frac{1}{\alpha+1}}\left(\int_{B_{r}}dx\right)^{\frac{\alpha}{\alpha+1}}\nonumber \leq \frac{C}{r}. \nonumber
\end{align}
It follows that
$$r^{2}e^{v(r)}\geq Cr.$$
By Jensen's inequality, we have

$$\max_{\partial B_{r}}(F^{0}(x)^{2}e^{u(x)})=r^{2}\max_{\partial B_{r}}e^{u(x)}\geq \frac{r^{2}}{N\kappa_{0}r^{N-1}}\int_{\partial B_{r}}e^{u}dS\geq r^{2}e^{v(r)}\geq Cr,$$
this is a contradiction.
\end{proof}

Here, we classify stable outside a compact set of $\mathbb{R}^{N}$ if $f(u)=u^{p}$ with $p=\frac{N+2}{N-2}$. When $F(\xi)=|\xi|$,  for the Laplacian operator,  such classification is established by  Farina in \cite{F}. For the quasilinear setting, Ciraolo-Figalli-Roncoroni in \cite{CFR} studied (\ref{1.1}) for $f(u)=u^{p}$ with the critical exponent. 

\begin{thm}\label{thm1.6}
If $f(u)=u^{p}$ with the critical exponent $p=\frac{N+2}{N-2}$,  then $u$ is a stable outside a compact set solution of (\ref{1.1}) in $\mathbb{R}^{N}$ if and only if
$$u(x)=\left(\frac{\lambda\sqrt{N(N-2)}}{\lambda^{2}+F^{0}(x-x_{0})^{2}}\right)^{\frac{N-2}{2}},$$
 for some $\lambda>0$ and $x_{0}\in\mathbb{R}^{N}$.
\end{thm}

\begin{proof}
It is discussed in \cite{CFR} that any positive weak solutions of equation (\ref{1.1}) are radial and  the form is $$u_{\lambda}(x)=\left(\frac{\lambda\sqrt{N(N-2)}}{\lambda^{2}+F^{0}(x-x_{0})^{2}}\right)^{\frac{N-2}{2}},$$
for some $\lambda>0$ and $x_{0}\in\mathbb{R}^{N}$. Next, we will claim that $u_{\lambda}(x)$ is stable outside a compact set. Assume $x_{0}=0$, we observe that $p|u_{\lambda}(x)|^{p-1}=O((F^{0}(x))^{-4})$ as $F^{0}(x)\rightarrow\infty$, therefore, we can find $R_{0}>0$ such that for any $F^{0}(x)>R_{0}$ we have
$$p|u_{\lambda}(x)|^{p-1}\leq\frac{(N-2)^{2}}{4}(F^{0}(x))^{-2} . $$
From  Theorem A, we have $F_{\xi}(\nabla u_{\lambda})=\frac{x}{F^{0}(x)}$. Hence
 \begin{align}
 &\int_{\Omega}F_{\xi_{i}}(\nabla u_{\lambda})F_{\xi_{j}}(\nabla u_{\lambda})\phi_{x_{i}}\phi_{x_{j}}+F(\nabla u_{\lambda})F_{\xi_{i}\xi_{j}}(\nabla u_{\lambda})\phi_{x_{i}}\phi_{x_{j}}-pu_{\lambda}^{p-1}\phi^{2}dx\nonumber\\
 &\geq\int_{\Omega}F_{\xi_{i}}(\nabla u_{\lambda})F_{\xi_{j}}(\nabla u_{\lambda})\phi_{x_{i}}\phi_{x_{j}}-pu_{\lambda}^{p-1}\phi^{2}dx\nonumber\\
 &=\int_{\Omega}\left|\frac{x}{F^{0}(x)}\cdot\nabla\phi\right|^{2}-pu_{\lambda}^{p-1}\phi^{2}dx\nonumber\\
 &\geq\int_{\Omega}\left|\frac{x}{F^{0}(x)}\cdot\nabla\phi\right|^{2}-\frac{(N-2)^{2}}{4}\frac{\phi^{2}}{F^{0}(x)^{2}}dx\geq0,\nonumber
 \end{align}
 the last inequality follows by  Proposition \ref{pro2} with $s=2$.  The desired result is proved.
\end{proof}

\section{Monotonicity Formulas } \label{Sec-MF}

Here we state the following monotonicity formulas for the equation (\ref{1.1}) with $f(u)$ is exponential-type and power-type nonlinearities. For the isotropic case, the following result $(i)$ is given in \cite{AY}, and $({ii})$ and  $({iii})$ are given by \cite{P} and \cite{GW},  respectively.

\begin{thm}\label{thm4}
Let $u\in H^{1,2}_{loc}(\mathbb{R}^{N})$ be a weak solution of (\ref{1.1}), for $x_{0}\in \mathbb{R}^{N}$ and $\lambda>0$, under the assumption of (\ref{1.7}).

 \begin{enumerate}
 \item [(i)]  If $f(u)=e^{u}$ and assume that $e^{u}\in L^{1}_{loc}(\mathbb{R}^{N})$ we define
\begin{align}\label{1.8}
E(u,x_{0},\lambda):=\lambda^{2-n}\int_{B_{\lambda}(x_{0})}\frac{1}{2}F(\nabla u)^{2}-e^{u} dx
+2\lambda^{1-n}\int_{\partial B_{\lambda}(x_{0})}(u+2\log\lambda)dS,
\end{align}
then $E(u,x_{0},\lambda)$ is a nondecreasing function of $\lambda$.
 \item [(ii)]  
If $f(u)=u^{p}$, and $u\in L^{p+1}_{loc}(\mathbb{R}^{N})$ we define
\begin{align}\label{1.9}
E_{1}(u,x_{0},\lambda):=\lambda^{\frac{2p+2}{p-1}-n}\int_{B_{\lambda}(x_{0})}\frac{1}{2}F(\nabla u)^{2}-\frac{1}{p+1}u^{p+1}dx
+\frac{1}{p-1}\lambda^{\frac{p+3}{p-1}-n}\int_{\partial B_{\lambda}(x_{0})}u^{2}(x)dS,
\end{align}
then $E_{1}(u,x_{0},\lambda)$ is a nondecreasing function of $\lambda$.
 \item [(iii)]  
If $f(u)=-u^{-p}$, and $u^{1-p}\in L^{1}_{loc}(\mathbb{R}^{N})$ we define
\begin{align}\label{1.10}
E_{2}(u,x_{0},\lambda):=\lambda^{\frac{2p-2}{p+1}-n}\int_{B_{\lambda}(x_{0})}\frac{1}{2}F(\nabla u)^{2}+\frac{1}{1-p}u^{1-p}dx
-\frac{1}{p+1}\lambda^{\frac{p-3}{p+1}-n}\int_{\partial B_{\lambda}(x_{0})}u^{2}(x)dS,
\end{align}
then $E_{2}(u,x_{0},\lambda)$ is a nondecreasing function of $\lambda$.

 \end{enumerate}

\end{thm}

\begin{proof}
$(i)$ For $f(u)=e^{u}$, define
$$E(\lambda):=\lambda^{2-n}\int_{B_{\lambda}(x_{0})}\frac{1}{2}F(\nabla u)^{2}-e^{u}dx. $$
Set
$u^{\lambda}(x)=u(\lambda x)+2\log\lambda$,
then we have
\begin{align}
E(\lambda)&=\lambda^{2-n}\int_{B_{\lambda}(x_{0})}\frac{1}{2}F(\nabla u)^{2}-e^{u}dx\nonumber\\
&=\lambda^{2-n}\int_{B_{1}(x_{0})}[\frac{1}{2}F(\nabla u^{\lambda})^{2}-e^{u^{\lambda}}]\lambda^{n-2} dy\nonumber\\
&=\int_{B_{1}(x_{0})}\frac{1}{2}F(\nabla u^{\lambda})^{2}-e^{u^{\lambda}} dy . \nonumber
\end{align}
It follows from Theorem \myref{A} and (\ref{1.7}), we have
\begin{align}
\frac{d}{d\lambda}E(\lambda)&=\int_{B_{1}(x_{0})}F(\nabla u^{\lambda})F_{\xi_{i}}(\nabla u^{\lambda})\frac{d}{d\lambda}\frac{\partial u^{\lambda}}{\partial x_{i}}-e^{u^{\lambda}}\frac{du^{\lambda}}{d\lambda} dy\nonumber\\
&=\int_{\partial B_{1}(x_{0})}F(\nabla u^{\lambda})F_{\xi_{i}}(\nabla u^{\lambda})\frac{du^{\lambda}}{d\lambda}\nu_{i}dS\nonumber\\
&=\int_{\partial B_{1}(x_{0})}F(\nabla u^{\lambda})\langle F_{\xi}(\nabla u^{\lambda}), F^{0}_{\xi}(y)\rangle\frac{du^{\lambda}}{d\lambda}dS\nonumber\\
&=\int_{\partial B_{1}(x_{0})}\langle \nabla u^{\lambda}, y\rangle\frac{du^{\lambda}}{d\lambda}dS\nonumber\\
&=\int_{\partial B_{1}(x_{0})}(\lambda\frac{du^{\lambda}}{d\lambda}-2)\frac{du^{\lambda}}{d\lambda}dS\nonumber\\
&=\int_{\partial B_{1}(x_{0})}\lambda(\frac{du^{\lambda}}{d\lambda})^{2}-2\frac{du^{\lambda}}{d\lambda}dS . \nonumber
\end{align}
Define
$$E(u,x_{0},\lambda):=\lambda^{2-n}\int_{B_{\lambda}(x_{0})}\frac{1}{2}F(\nabla u)^{2}-e^{u} dx
+2\lambda^{1-n}\int_{\partial B_{\lambda}(x_{0})}(u+2\log\lambda)dS,$$
therefore, we have
$$\frac{d}{d\lambda}E(u,x_{0},\lambda)\geq0.$$

$(ii)$ If $f(u)=u^{p}$, then 
$$E_{1}(\lambda):=\lambda^{\frac{2p+2}{p-1}-N}\int_{B_{\lambda}(x_{0})}\frac{1}{2}F^{2}(\nabla u)-\frac{1}{p+1}u^{p+1}dx . $$
Set $u_{\lambda}(x)=\lambda^{\frac{2}{p-1}}u(\lambda x)$, we know that
$$E_{1}(\lambda)=\int_{B_{1}(x_{0})}\frac{1}{2}F^{2}(\nabla u_{\lambda}(y))-\frac{1}{p+1}u_{\lambda}(y)^{p+1}dy . $$
Hence, it follows from Theorem \myref{A} and (\ref{1.7}), we have
\begin{align}
\frac{d}{d\lambda}E_{1}(\lambda)&=\int_{B_{1}(x_{0})}F(\nabla u_{\lambda}(y))F_{\xi_{i}}(\nabla u_{\lambda}(y))\frac{d}{d\lambda}\frac{\partial u_{\lambda}(y)}{y_{i}}-u_{\lambda}(y)^{p}\frac{du_{\lambda}(y)}{d\lambda}dy\nonumber\\
&=\int_{\partial B_{1}(x_{0})}F(\nabla u_{\lambda}(y))F_{\xi_{i}}(\nabla u_{\lambda}(y))\frac{du_{\lambda}(y)}{d\lambda}\nu_{i}dS\nonumber\\
&=\int_{\partial B_{1}(x_{0})}\langle \nabla u_{\lambda}(y),y\rangle\frac{du_{\lambda}(y)}{d\lambda}dS\nonumber\\
&=\int_{\partial B_{1}(x_{0})}[\lambda\frac{du_{\lambda}(y)}{d\lambda}-\frac{2}{p-1}u_{\lambda}(y)]\frac{du_{\lambda}(y)}{d\lambda}dS\nonumber\\
&=\int_{\partial B_{1}(x_{0})}\lambda(\frac{du_{\lambda}(y)}{d\lambda})^{2}-\frac{1}{p-1}\frac{du^{2}_{\lambda}(y)}{d\lambda}dS.\nonumber
\end{align}

Define
\begin{align}
E_{1}(u,x_{0},\lambda):&=\int_{B_{1}(x_{0})}\frac{1}{2}F^{2}(\nabla u_{\lambda}(y))-\frac{1}{p+1}u_{\lambda}(y)^{p+1}dy
+\int_{\partial B_{1}(x_{0})}\frac{1}{p-1}u^{2}_{\lambda}(y)dS\nonumber\\
&=\lambda^{\frac{2p+2}{p-1}-N}\int_{B_{\lambda}(x_{0})}\frac{1}{2}F^{2}(\nabla u)-\frac{1}{p+1}u^{p+1}dx +\frac{1}{p-1}\lambda^{\frac{p+3}{p-1}-N}\int_{\partial B_{\lambda}(x_{0})}u^{2}(x)dS . \nonumber
\end{align}
Therefore, we have
$$\frac{d}{d\lambda}E_{1}(u,x_{0},\lambda)\geq0.$$

$(iii)$ Let $f(u)=-u^{-p}$, then 
$$E_{2}(\lambda):=\lambda^{\frac{2(p-1)}{p+1}-N}\int_{B_{\lambda}(x_{0})}\frac{1}{2}F^{2}(\nabla u)+\frac{1}{1-p}u^{1-p}dx . $$
Set $u_{\lambda}(x)=\lambda^{\frac{2}{p-1}}u(\lambda x)$, we know that
$$E_{2}(\lambda)=\int_{B_{1}(x_{0})}\frac{1}{2}F^{2}(\nabla u_{\lambda}(y))+\frac{1}{1-p}u_{\lambda}(y)^{1-p}dy. $$
Hence, it follows from Theorem \myref{A} and (\ref{1.7}), we have
\begin{align}
\frac{d}{d\lambda}E_{2}(\lambda)&=\int_{B_{1}(x_{0})}F(\nabla u_{\lambda}(y))F_{\xi_{i}}(\nabla u_{\lambda}(y))\frac{d}{d\lambda}\frac{\partial u_{\lambda}(y)}{y_{i}}+u_{\lambda}(y)^{-p}\frac{du_{\lambda}(y)}{d\lambda}dy\nonumber\\
&=\int_{\partial B_{1}(x_{0})}F(\nabla u_{\lambda}(y))F_{\xi_{i}}(\nabla u_{\lambda}(y))\frac{du_{\lambda}(y)}{d\lambda}\nu_{i}dS\nonumber\\
&=\int_{\partial B_{1}(x_{0})}\langle \nabla u_{\lambda}(y),y\rangle\frac{du_{\lambda}(y)}{d\lambda}dS\nonumber\\
&=\int_{\partial B_{1}(x_{0})}[\lambda\frac{du_{\lambda}(y)}{d\lambda}+\frac{2}{p+1}u_{\lambda}(y)]\frac{du_{\lambda}(y)}{d\lambda}dS\nonumber\\
&=\int_{\partial B_{1}(x_{0})}\lambda(\frac{du_{\lambda}(y)}{d\lambda})^{2}+\frac{1}{p+1}\frac{du^{2}_{\lambda}(y)}{d\lambda}dS.\nonumber
\end{align}

Define
\begin{align}
E_{2}(u,x_{0},\lambda):&=\int_{B_{1}(x_{0})}\frac{1}{2}F^{2}(\nabla u_{\lambda}(y))+\frac{1}{1-p}u_{\lambda}(y)^{1-p}dy -\int_{\partial B_{1}(x_{0})}\frac{1}{p+1}u^{2}_{\lambda}(y)dS\nonumber\\
&=\lambda^{\frac{2p-2}{p+1}-N}\int_{B_{\lambda}(x_{0})}\frac{1}{2}F^{2}(\nabla u)+\frac{1}{1-p}u^{1-p}dx-\frac{1}{p+1}\lambda^{1-N-\frac{4}{p+1}}\int_{\partial B_{\lambda}(x_{0})}u^{2}(x)dS . \nonumber
\end{align}
Therefore, we have
$$\frac{d}{d\lambda}E_{2}(u,x_{0},\lambda)\geq0,$$
This completes the proof. 
\end{proof}

\section*{Acknowledgements}
The second author would like to thank University of Texas at San Antonio for hospitality during his visits.

\end{document}